\documentclass{amsproc}
\usepackage[abbrev,alphabetic,nobysame,
]{amsrefs}

\usepackage{graphicx}
\usepackage{enumerate}
\usepackage{xypic}

\makeatletter
\let\@cleartopmattertags\relax
\newcommand\articleend
 {\enddoc@text
  \let\authors\@empty
  \let\contribs\@empty
  \let\xcontribs\@empty
  \let\toccontribs\@empty
  \let\addresses\@empty
  \let\thankses\@empty
  \let\dedicatory\@empty
\newpage
 }
\let\@wraptoccontribs\wraptoccontribs
\makeatother

\newtheorem{theorem}{Theorem}[section]
\newtheorem{lemma}[theorem]{Lemma}
\newtheorem{prop}[theorem]{Proposition}
\newtheorem{corollary}[theorem]{Corollary}

\theoremstyle{definition}
\newtheorem{definition}[theorem]{Definition}
\newtheorem{example}[theorem]{Example}

\theoremstyle{remark}
\newtheorem{remark}[theorem]{Remark}

\numberwithin{equation}{section}

\newcommand{\F}{\mathbb F}
\newcommand{\Z}{\mathbb Z}
\newcommand{\Q}{\mathbb Q}
\newcommand{\R}{\mathbb R}

\renewcommand{\a}{\mathbf{a}}
\renewcommand{\b}{\mathbf{b}}
\renewcommand{\c}{\mathbf{c}}

\renewcommand{\phi}{\varphi}

\def\epsilon{\varepsilon}
\def\kappa{\varkappa}

\renewcommand{\pod}[1]{\if@display\mkern10mu\else\mkern6mu\fi(#1)}
\renewcommand{\pmod}[1]{\pod{{\operator@font mod}\mkern6mu#1}}

\usepackage[left=1.5cm,right=1.5cm, top=1.5cm,bottom=1.5cm,bindingoffset=0cm]{geometry}

\newcounter{ineqone}
\newcommand{\ione}{\left[\arabic{ineqone}\right] \addtocounter{ineqone}{1}}

\begin{document}
\allowdisplaybreaks

\title{Groups of points on abelian threefolds over finite fields}

\author{Yulia Kotelnikova}

\address{Institute for Information Transmission Problems (RAS), Moscow, Russia}
\address{National Research University “Higher School of Economics”, Moscow, Russia}




\email{yuliakotel@gmail.com}

\keywords{abelian variety, the group of rational points, finite field, Littlewood-Richardson rule}


%
%

\begin{abstract}
In this paper we provide an algorithm to classify groups of points on abelian threefolds over finite fields.
The classification is given in terms of the Weil polynomial of abelian varieties in a given $\mathbb{F}_q$-isogeny class. This work completes  partial classification given in \cite{Ry15}.
\end{abstract}

\maketitle

\section{Introduction}

Famous results of Tsfasman \cite{tsf} and Xing \cite{xi1}, \cite{xi2} concerning groups of $\F_q$--points on elliptic curves and abelian surfaces  were elegantly generalized by Rybakov in \cite{Ry10} and \cite{Ry12}. This paper is devoted to  classification of groups of points on abelian varieties of dimension $3$. The classification was started  in \cite{Ry15}. In this section we site some of Rybakov's notable theorems and  give a precise problem statement.

Suppose $X$ is an abelian variety of dimension $g$ over a finite field  $\F_q$ with $q=p^r$. Then $X(\F_q)$ is a finite abelian group. Thus there is a decomposition $ X(\F_q)=\oplus_lX(\F_q)_l$
into a direct sum of its $l$-primary components. 

We denote by $\operatorname{T}_l(X)=\underleftarrow{\lim}\ker(X\xrightarrow{\:\cdot\, l^m}X)(\overline\F_q)$ the  \emph{Tate module} of $X$. In case $l\ne p$ it  is a $\Z_l$ module of rank $2g$.  We denote by $\operatorname{Fr}_X$ the Frobenius map acting on $\operatorname{T}_l(X)$.  The characteristic polynomial $f_X(t)$ of  $\operatorname{Fr}_X$ is a \emph{q-Weil polynomial}, i.\,e. $f_X(t)$ is monic, defined over $\Z$ and absolute values of its roots are all equal to~$\sqrt{q}$. 

The following statement holds:
\begin{prop}[\cite{Ry10}]
The group of points $X(\F_q)_l$ is isomorphic to $\operatorname{coker}(1-\operatorname{Fr}_X)$.
\end{prop}
We call the exponents of the group  $\operatorname{coker}(1-\operatorname{Fr}_X)$ Smith's invariants of the operator $(1-\operatorname{Fr}_X)$ as well as of a $\Z_l[1-\operatorname{Fr}_X]$--module $\operatorname{T}_l(X)$ (see section \ref{horn} for a thorough definition).

To state the results by Rybakov concerning groups of points on ordinary varieties it is convenient to introduce the notion of Newton polygon and Hodge polygon. Both of them are thought of as subsets of plane with coordinates $(x,y)$.
Suppose $G$ is an $l$-group isomorphic to a direct sum
$$G=\Z/l^{c_1}\oplus\Z/l^{c_2}\oplus\ldots\oplus \Z/l^{c_d}$$
where $c_1\ge c_2\ge\ldots\ge c_d\ge 0$. We consider the set of  points $(i,\sum_{j=d-i}^{d}c_j)$ and  define the Hodge polygon $\operatorname{Hp}_l(G,d)$  of this group to be the lower boundary of its convex hull.

Let $P(t)\in \Z_l[t]$ be a monic polynomial which over $\overline{\Q_l}$ decomposes into a product
$$P(t)=(t-\alpha_1)(t-\alpha_2)\ldots(t-\alpha_d),$$
and suppose $v_l(\alpha_1)\ge v_l(\alpha_2)\ldots \ge v_l(\alpha_d)\ge 0$. We consider the set of  points $(i,\sum_{j=d-i}^{d}v_l(\alpha_j))$ and  define the Newton polygon $\operatorname{Np}(P(t))$ of this polynomial to be the lower boundary of  its convex hull. Alternatively, if 
$$P(t)=t^d+f_{1}t^{d-1}+\ldots+f_d,$$
then $\operatorname{Np}(P(t))$ is the lower boundary of the convex hull of the set $(i,v_l(f_i))$.

Now we can state two beautiful theorems:
\begin{theorem}[\cite{Ry10}]\label{maintheo}
 $\operatorname{Np}(f_X(1-t))$   lies on or above  $\operatorname{Hp}(A(\F_q)_l, 2g)$  and their endpoints coincide. 

Explicitly, if 
$$\begin{array}{ll}A(\F_q)_l=\Z/l^{c_1}\oplus\Z/l^{c_2}\oplus\ldots \oplus\Z/l^{c_{2g}},&c_1\ge c_2\ge\ldots\ge c_{2g}\ge 0;\\
f(1-t)=(t-\alpha_1)(t-\alpha_2)\ldots(t-\alpha_{2g}),\qquad&v_l(\alpha_1)\ge v_l(\alpha_2)\ldots \ge v_l(\alpha_{2g})\ge 0,\end{array}$$
then the following inequalities hold
\begin{equation}\begin{array}{l}\label{up}
v_l(\alpha_1)\ge c_1;\\
v_l(\alpha_1)+v_l(\alpha_2)\ge c_1+c_2;\\
\ldots\\
v_l(\alpha_1)+v_l(\alpha_2)+\ldots+v_l(\alpha_{2g})= c_1+c_2+\ldots+c_{2g}.
\end{array}\end{equation}
\end{theorem}

\begin{theorem}[\cite{Ry10}]\label{maintheo2}
Suppose $f_X(t)$  is separable  and given an $l$-group $H$ such that\\   $\operatorname{Np}(f_X(1-t))$   lies on or above  $\operatorname{Hp}(H, 2g)$  and their endpoints coincide. 
 Than there exists an abelian variety $\tilde X$ isogenous to~$X$ such that $\tilde X(\F_q)_l\cong H$. 
Equivalently, suppose $f_X(t)$  is separable  and  suppose
$$H=\Z/l^{c_1}\oplus\Z/l^{c_2}\oplus\ldots \oplus\Z/l^{c_{2g}},\qquad c_1\ge c_2\ge\ldots\ge c_{2g}\ge 0,$$
so that  inequalities \ref{up} hold. Than there exists an abelian variety $\tilde X$ isogenous to~$X$ such that $\tilde X(\F_q)_l\cong H$. 
\end{theorem}

The case of nonseparable $f_X$ (equivalently, of $X$ having noncommutative endomorphism algebra) turned up to  be tricky. Tsfasman described groups of points for  $\dim X=1$ in \cite{tsf}. The classification for $\dim X=1,\,2$ and partially in case $\dim X=3$ was obtained by Rybakov in  \cite{Ry10}, \cite{Ry12}, and \cite{Ry15}. In this work we are going to apply the following result:

\begin{theorem}[\cite{Ry12}; \cite{Ry15}]\label{groupsum}
Suppose $f_X(1-t)=P^2(t)$, $\deg P(t)=2$ and $P(t)$ is separable. Then the group $X(\F_q)_l$ is decomposable into a direct sum of two groups
$$X(\F_q)_l=H_1\oplus H_2$$ each having at most two generators such that $\operatorname{Np}(P(1-t))$ lies on or above  $\operatorname{Hp}(H_i, 2)$ and their right borders coincide for $i=1,2$.
Explicitly, suppose
$$\begin{array}{ll}
A(\F_q)_l=\Z/l^{c_1}\oplus\Z/l^{c_2}\oplus\Z/l^{c_3}\oplus\Z/l^{c_4},\qquad& c_1\ge\ldots\ge c_4\ge 0,\\
P(t)=(t-\alpha_1)(t-\alpha_2),\qquad&v_l(\alpha_1)\ge v_l(\alpha_{2})\ge 0,\end{array}
$$
then
$$\begin{array}{l}
v_l(\alpha_1)\ge c_1\ge c_2,\\
c_1+c_4=c_2+c_3=v_l(\alpha_1)+v_l(\alpha_2).\qquad
\end{array}
$$
\end{theorem}

In the present work we complete the classification of groups of points on abelian threefolds. The following three situations are discussed:
\begin{equation}\begin{array}{llll}
(1)&f_X(t)=P^2(t)Q(t),&\deg P=\deg Q=2,\quad&P(t)Q(t)\mbox{ separable};\\
(2)&f_X(t)=P(t)(t\pm\sqrt{q})^2,&\deg P=4,&P(t)(t\pm\sqrt{q})\mbox{ separable};\\
(3)&f_X(t)=Q^2(t)(t\pm\sqrt{q})^2,\quad&\deg Q=2,&Q(t)(t\pm\sqrt{q})\mbox{ separable}.
\end{array}\label{problem}\end{equation}
The concept  is to include the desired Tate module $\operatorname{T}_l(X)$ in an exact sequence of $\Z_l[T]$--modules
$$0\longrightarrow \operatorname{T}_l(Y)\longrightarrow \operatorname{T}_l(X)\longrightarrow \operatorname{T}_l(Z) \longrightarrow 0,$$
so that both $Y$ and $Z$ have dimension at most $2$. In this situation Smith's invariants of $T$ acting as one minus the Frobenius map on $\operatorname{T}_l(Y)$ and $\operatorname{T}_l(Z)$ are known by Rybakov's theorems. Due to the existence of the exact sequence, the set of Smith's invariants of $\operatorname{T}_l(X)$ taken together with those of  $\operatorname{T}_l(Y)$ and $\operatorname{T}_l(Z)$ form a Liouville-Richardson triple (that is, the set of Smith's invariants of $\operatorname{T}_l(X)$ can be obtained from those of  $\operatorname{T}_l(Y)$ and $\operatorname{T}_l(Z)$ by applying the Liouville-Richardson rule).
Such triples appear in many problems, for instance, if one wants to find a decomposition of a $\mathfrak{gl}_n$--module $V_{\lambda}\otimes V_{\mu}$ into a sum of irreducible modules. Some other applications are listed in the Section \ref{onhorn}. Finally, Liouville-Richardson triples are exactly the ones satisfying certain set of inequalities.

The fact that Smith's invariants of modules in an exact triple form a Liouville-Richardson triple was discovered J.\,A\,Green, T.\,Klein; the proof can be found in \cite{Mac}. In the papers  \cite{th1} and  \cite{sqs} divisibility conditions on the Smith's invariants are obtained from the theorem by Green and Klein. These results are surveyed in  \cite{Fu}, especially Section 2. The idea to apply these results to the problem of finding groups of points on abelian varieties was introduced in \cite{me} but both main theorems of this paper were incorrect.

We describe this set of inequalities in Section \ref{onhorn} following W.\,Fulton's survey~\cite{Fu}. In \ref{sumsof} we recall theorems on eigenavalues of sums of Hermitian matrices by  A.\,Knutson and T.\,Tao in~\cite{KT}. The aim of Section~\ref{horn} is to state the results concerning Smith's invariants using the definitions and theorems of  Section~\ref{sumsof} as well as examples given in Section~\ref{examples}.  In Section  \ref{horn} we also shorten the list of inequalities thus we obtain  Corollary~\ref{cor4} from Theorem \ref{theo4}.

In Section \ref{compu}  we explain thoroughly in what way  Corollary \ref{cor4} is applied to the problem of computation of groups of points on threefolds with Weil polynomials from the list \ref{problem}. In Sections \ref{dif ficult} and  \ref{easy} we conduct the computations and obtain the lists of inequalities. In Section \ref{mainresult} the results are summarized in three theorems concerning each Weil polynomial from the list \ref{problem}.

\section{On Horn conjecture}	 	\label{onhorn}

Relation between sets of Smith's invariants (invariant factors) of a pair of matrices over a (local principal ideal) domain~$R$ and one of their product is described by the conditions appearing also in  Horn conjecture. Interpretation in terms of Hermitian matrices provides some tools for dealing with these conditions which turns up to be helpful when it comes to Smith's invariants (see Example \ref{dual} and Lemma \ref{cor1}). In the first half of this section we cite celebrated results on sums of Hermitian matrices, and  the second half is devoted to Smith's invariants.

\subsection{Sums of Hermitian matrices}\label{sumsof} The first part of this section is dedicated to manipulations with hermitian matrices. We follow the text \cite{Fu} in which the results of \cite{KT} are surveyed.  

Recall that a hermitian operator is diagonalizable, all its  eigenvalues  are real and eigenspaces corresponding to distinct eigenvalues are orthogonal to each other. Suppose $(A, B, C=A+B)$ is a triple of hermitian $n\times n$ matrices acting on a vector space $V$ with eigenvalues $\a=(a_1\ge a_2\ge \ldots\ge a_n)$, $\b=(b_1\ge b_2\ge \ldots\ge b_n)$, $\c=(c_1\ge c_2\ge \ldots\ge c_n)$ respectively. It turns up that there is a list of inequalities of the form
$$\sum_{i\in I}a_i+\sum_{j\in J} b_j\ge \sum_{k\in K} c_k \leqno(*_{IJK}) $$
with $I$, $J$, $K$ certain subsets of $\{1, 2,\ldots , n\}$ giving  necessary and sufficient conditions for a triple of hermitian matrices $A$, $B$, $C=A+B$ with eigenvalues $\a$, $\b$, $\c$ to exist. We shall now describe the triples $(I,J,K)$ such that the corresponding inequalities  $(*_{IJK})$ belong to the list. 

For this purpose we introduce some more notation. A subscript \emph{always} denotes number of elements in the set, whereas superscripts may denote various additional information on elements of the set such as upper or lower bounds. In case an expression is taken in parentheses, the superscript denotes the degree.
$$\begin{array}{l}
M_n=\{1, 2, \ldots, n\};\\
\Lambda_{\:\:p}^ {\le n}=\left\{\,(i_1<i_2<\ldots< i_p)\in (\Z_{\ge 0})^p\quad|\quad 1\le i_1;\quad i_p\le n\,\right\};\\
U_p^n=\Big\{\,(I,J,K)\in \left(\Lambda_{\:\:p}^{\le n}\right)^3\quad|\quad
\sum_{i\in I} i+\sum_{j\in J}j=\sum_{k\in K} k+\frac{p(p+1)}2\,\Big\}.
\end{array}$$
Desired triples are organized in sets $T^n_p\subset U_p^n$. First, we set  
$T^n_1=U_1^n$. 
The sets $T^n_p$ 
are now defined recursively by
$$\begin{array}{r}T^n_p=\Big\{\,(I,J,K)\in U_p^n \quad \vline\quad \forall\;1\le r <p \quad\forall\;(F,G,H)\in T^p_r\\
\sum_{f\in F}i_f+\sum_{g\in G} j_g\ge \sum_{h\in H} k_h +\frac{r(r+1)}2\,\Big\}.\end{array}$$

We also introduce sets 
$$\begin{array}{l}
L_n^{\R}=\{(\lambda_1\ge \lambda_2\ge\ldots\ge\lambda_n)\}\subset (\R)^n;\\
L_{n}^{+}=\{(\lambda_1\ge \lambda_2\ge\ldots\ge\lambda_n)|\lambda_n\ge 0\}\subset (\Z_{\ge 0 })^n.
\end{array}$$

Now we can state three main theorems concerning eigenvalues of sums of Hermitian matrices.

\begin{theorem}[\cite{Fu}, Theorem 1]  \label{theo1}
Let $(\a,\b,\c)$ be a triple in  $(L_n^\R)^3$. There exist a triple $A$,  $B$, $C=A+B$ of hermitian $n\times n$ matrices with eigenvalues $\a$, $\b$, and $\c$ respectively if and only if 
$$\sum_{\a}a+\sum_\b b=\sum_\c c$$
and $(*_{IJK})$ holds for all $(I,J,K)\in T^n_p$ for each $p$. 
\end{theorem}

\begin{theorem}[\cite{Fu}, Theorem 2] 
Let $A$,  $B$, $C=A+B$  be a triple of hermitian $n\times n$ matrices with eigenvalues $(\a,\b,\c)\in (L_n^\R)^3$ and let
$$\sum_{i\in I} a_i+\sum_{j\in J} b_j=\sum_{k\in K} c_k$$
 for some triple $(I,J,K)\in T^n_p$. Then  there is a $p$-dimensional subspace $W\subset V$ invariant under the action of $A$, $B$ and $C$. 
\end{theorem}
Due to orthogonality of eigenspaces this in fact implies that the matrices $A$, $B$ and $C$ have a couple of complementary common invariant subspaces in $V$.
 
We define a map $\lambda\,:\, \Lambda^{\le n}_{\:\:p}\to L_{p}^+$ as follows: 
$$\lambda(I)=(i_p-p\ge i_{p-1}-p+1\ge \ldots\ge i_1-1).$$ 

\begin{theorem}[\cite{Fu}, Theorem 5] \label{theo3}
Let  $(I, J,K)$ be a triple in $U_p^n$. There exist a triple $A$,  $B$, $C=A+B$ of hermitian matrices with eigenvalues $\lambda(I)$, $\lambda(J)$, and $\lambda(K)$ respectively if and only if  $(I,J,K)\in T^p_n$.
\end{theorem}

Hence if  $(I,J,K)\in T^n_p$ one may consider some triple of hermitian matrices with  eigenvalues $\lambda(I)$, $\lambda(J)$, and $\lambda(K)$. We shall (abusively) write  $A(I)$, $B(J)$, and $C(K)$ meaning that we pick up any triple, although there might exist many of them. For example, consider a triple $(I,J,K)\in T^n_p$ with  $J=(1, 2, \ldots, p)$.  In this case we say that  $B(J)=0$ and $A(I)=C(K)$, even though the latter 
ones are not specified.

\begin{remark}
It turns up that the inequalities $(*_{IJK})$ for all $1\le p\le n-1$ and all $(I,J,K)\in T^n_p$ together with
\begin{equation}\begin{array}{l}\label{up1}
a_1\ge a_2\ge \ldots \ge a_n\\
b_1\ge b_2\ge \ldots \ge b_n\\
c_1\ge c_2\ge \ldots \ge c_n\\
\sum_{a\in\a} a+\sum_{b\in\b}=\sum_{c\in\c}
\end{array}\end{equation}
are not independent for $n\ge 6$. In particular, for $n=6$ the only redundant inequality is
$$a_1+a_3+a_5+b_1+b_3+b_5\ge c_2+c_4+c_5.$$
This inequality follows from \ref{up1}. However, in practice, Lemma \ref{cor3} and Corollary \ref{cor1} leaves this particular inequality out of consideration. See also \cite{Fu}, Example 1. 
\end{remark}

\subsection{Examples}\label{examples} We give some examples of triples in $T^n_p$.

\begin{example}
By definition, $T^n_1$ consists of triples
$$(\{i\}, \{j\},\{k=i+j-1\}).$$
It is not hard to describe explicitly the set $T^n_2$. A triple $(I,J,K)\in U_2^n$ belongs to $T_2^n$ if and only if 
$$\begin{array}{l}
i_1+j_1\le k_1+1,\\
i_2+j_1\le k_2+1,\\
i_1+j_2\le k_2+1.
\end{array}$$
See also enlightening \cite{Ho} theorem $8$ and also theorem $9$ for the explicit description of $T^n_3$.
\end{example}

\begin{example}{\it Complementary ineq\label{dual}ualities.}
One can look for complementary inequalities of the form
$$\sum_{i\in I^c}a_i+\sum_{j\in J^c} b_j\le \sum_{k\in K^c} c_k$$
by subtracting $(*_{IJK})$ from the equality
$$\sum_{\a}a+\sum_{\b} b= \sum_{\c} c.$$
For example, the set $T^n_{n-1}$ consists of triples 
$$(M_n\backslash\{i\}, M_n\backslash\{j\}, M_n\backslash\{i+j-n\})$$
which give the inequalities
$$a_i+b_j\le c_{i+j-n}.$$
Suppose $(I,J,K)\in T^n_p$. One obtains a complementary triple from matrices $n-A(I)$, $n-B(J)$, $n-C(K)$.
\end{example}

\begin{example}\label{lid1}{\it Lidskii inequalities.} Let  $J=\{1,2,\ldots, p\}$. The inequalities 
$i_\alpha+j_1\le k_\alpha+1$ for all $\alpha$ combined with 
$$\sum_{i\in I} i+\sum_{j\in J} j= \sum_{k\in K} k+\frac{p(p+1)}2$$
imply that $I=K$. Thus we have 
$$\sum_{i\in I} a_i+\sum_{j=1}^p b_j=\sum_{i\in I} c_i$$
As discussed above, in this situation $B(J)=0$ and $A(I)=C(K)$.
\end{example}

\begin{example}\label{lid2}
The previous example can be easily improved in a following way. Suppose $J=\{s+1,s+2,\ldots, s+p\}$. Again the inequalities 
$i_\alpha+j_1\le k_\alpha+1$ for all $\alpha$ combined with 
$$\sum_{i\in I} i+\sum_{j\in J} j= \sum_{k\in K} k+\frac{p(p+1)}2$$
imply that $i_\alpha=k_\alpha+s$ for all $\alpha$. In this case $B(J)$ is a scalar matrix.
\end{example}

\subsection{Smith's invariants} \label{horn}
For our purposes we will need the versions of these statements with Smith's invariants. For simplicity let  $R=\Z_l$ and let  $\mathcal{A}\in \operatorname{Mat}_{s\times s}(R)$. Then $R^s/\mathcal{A}R^s$ is a finite abelian $l$-group, hence it has the form 
$$R^s/\mathcal{A}R^s=\Z/l^{a_1}\Z\oplus\Z/l^{a_2}\Z\oplus\ldots\oplus\Z/l^{a_s}\Z$$
with $a_1\ge a_2\ge \ldots \ge a_s\ge 0$. The set $\mathfrak{a}=\{a_1\ge a_2\ge \ldots \ge a_s\}$ is said to be Smith's invariants of the matrix $\mathcal{A}$, as well as of the group $R^s/\mathcal{A}R^s$. 

Suppose there is an exact sequence of $R[T]$--modules
$$0\longrightarrow R^s\longrightarrow R^{s+t}\longrightarrow R^t\longrightarrow 0$$
with the action of $T$ on $R^s$, $R^t$, and $R^{s+t}$ defined by matrices $\mathcal A$, $\mathcal B$, and $\mathcal C$ respectively. In this case $\mathcal{C}$ is of the form
$$\mathcal{C}=\left(\begin{array}{ll}
			\mathcal A&\mathcal X\\
			0&\mathcal{B}
\end{array}\right)$$
for some $\mathcal{X}\in \operatorname{Mat}_{t\times s}(R)$. Moreover,  from

$$
\xymatrix{
	0\ar[r]&R^s\ar[r]\ar[d]_T &R^{s+t}\ar[r]\ar[d]_T &R^t\ar[r]\ar[d]_T\ar[r]&0\\
	0\ar[r]&R^s\ar[r] &R^{s+t}\ar[r] &R^t\ar[r]\ar[r]&0
}
$$
by the snake lemma we obtain  
$$0\longrightarrow \left.R^s_{\phantom{f}}\right/\mathcal{A}R^s\longrightarrow \left.R^{s+t}\right/\mathcal{C}R^{s+t}\longrightarrow \left.R^t\right/\mathcal{B}R^t\longrightarrow 0.$$
We shall now state the theorem which relates invariant factors of certain triples of matrices. 
\begin{theorem}[\cite{Fu}, Theorem 10; \cite{sqs}, Corollary 3.2]  \label{theo4}
Let  $(\mathfrak{a},\mathfrak{b},\mathfrak{c})$ be a triple in $L_s^+\times L_{t}^+\times L_{s+t}^+$. Then there exist such matrices $\mathcal{A}\in \operatorname{Mat}_{s\times s}(R)$, $\mathcal{B}\in \operatorname{Mat}_{t\times t}(R)$, $\mathcal{X}\in \operatorname{Mat}_{t\times s}(R)$ and 
$$\mathcal{C}=\left(\begin{array}{ll}
			\mathcal A&\mathcal X\\
			0&\mathcal{B}
\end{array}\right)$$
such that $\mathfrak{a}$, $\mathfrak{b}$, $\mathfrak{c}$ are Smith's invariants of $\mathcal A$, $\mathcal B$, $\mathcal C$ respectively if and only if
$$\sum_{a\in\mathfrak a} a+\sum_{b\in\mathfrak b} b=\sum_{c\in\mathfrak c}c$$
and for all $p$ and all $(I,J,K)\in T^n_p$
\begin{equation}
\sum_{i\in I\cap M_s} a_i+\sum_{j\in J\cap M_t} b_j\ge\sum_{k\in K}c_k.\label{lab}
\end{equation}
\end{theorem}
\begin{remark}
One can easily extend the above theorem to the case of arbitrary principal ideal domain. 
\end{remark}

The list \ref{lab} of inequalities  to check  contains redundant ones. For example, conditions
$$\sum_{i\in I\cap M_s} a_i+\sum_{j\in J\cap M_t} b_j\ge\sum_{k\in K}c_k$$
for $I=J=K=\{1,2,\ldots, p\}$ for $p>s,t$ follow from
$$\sum_{a\in\mathfrak a} a+\sum_{b\in\mathfrak b} b=\sum_{c\in\mathfrak c}c.$$
The lemmas below imply that list of inequalities  of the form \ref{lab}  can in fact be significantly shortened.

\begin{lemma}\label{cor1}
Suppose $(I, J,K)\in T^n_p$ and there is an $\alpha$ such that $i_{\alpha-1}+1<i_{\alpha}$ (or $\alpha = 1$ and $1<i_1$). Then there exists $\beta$  such that 
\begin{enumerate}
\item[$(1)$]$k_{\beta-1}+1<k_{\beta}$ (or $\beta=1$ and $1<k_\beta$);
\item[$(2)$]$(I', J,K')\in T^n_p$ with $I'=I\backslash \{i_{\alpha}\}\sqcup \{i_{\alpha}-1\}$, and $K'=K\backslash \{k_{\beta}\}\cup \{k_{\beta}-1\}$.
\end{enumerate}
\end{lemma} 
First, we give a restatement of the lemma:
\newtheorem*{lemm}{Lemma \ref{cor1}'}

\begin{lemm}Suppose $A$, $B$, $C=A+B$ are $p\times p$ matrices with integer nonnegative eigenvalues and 
$\a$, $\b$, and $\c$ respectively and suppose $a_{\alpha}\ne 0$ for some $\alpha\in M_p$. Then there exist
\begin{enumerate}
\item[$\tilde{(1)}$]  $\beta\in M_p$ such that $c_{\beta}\ne0$
\item[$\tilde{(2)}$] a triple  of matrices $A'$, $B'$, $C'=A'+B'$ with eigenvalues  $\a\backslash \{a_{\alpha}\}\sqcup \{a_{\alpha}-1\}$, $\b$, and $\c\backslash \{c_{\beta}\}\cup \{c_{\beta}-1\}$ respectively.
\end{enumerate}
\end{lemm}
\begin{proof}[Proof of Lemmas \ref{cor1} and \ref{cor1}']
Theorem \ref{theo3} implies that Lemma \ref{cor1} follows from Lemma \ref{cor1}' applied to $(A,B,C)=(A(I),B(J), C(K))$.  Thus it suffices to prove the later one.

Suppose there is no $\beta$  for which $\tilde{( 1)}$ holds. Then $C=0$ and therefore $A=B=0$ which contradicts the definition of $\alpha$.

We prove the lemmas by induction in $p$. If  $p=1$, then the statement holds.

Let matrices $(A,B,C)$ have dimensions $p>1$. Recall that by Theorem \ref{theo1} a triple $(A,B,C)$ satisfies certain list of inequalities $(*_{FGH})$. If there is no inequality $(*_{FGH})$ occurring with equality then for any $\beta$ satisfying $\tilde{(1)}$  the sets $\a\backslash \{a_{\alpha}\}\sqcup \{a_{\alpha}-1\}$, $\b$, and $\c\backslash \{c_{\beta}\}\cup \{c_{\beta}-1\}$ still satisfy conditions of Theorem \ref{theo1}. Therefore there exists a desired triple of matrices $(A',B',C')$.

Otherwise the matrices $A$, $B$, $C$ have a pair of common invariant complementary proper subspaces $W$ and $W^{\bot}$. Let  the set of eigenvalues of $A|_W$  contain $a_{\alpha}$. Then the induction step is applicable to the triple $(A|_W,B|_W,C|_W)$, thus there exists a triple which we denote by $((A|_W)',(B|_W)',(C|_W)')$. Then the desired triple of matrices is  $((A|_W)'\oplus A|_{W^{\bot}},(B|_W)'\oplus B|_{W^{\bot}},(C|_W)'\oplus C|_{W^{\bot}})$.

\end{proof}

\begin{definition}\label{up11}
Suppose $(I,J,K)\in  T^{s+t}_p $ and set
$$\begin{array}{l}\tilde I=I\cap M_s,\\
\tilde J= J\cap M_t.\end{array}$$ 
We say that $(I,J,K)\in  \tilde{T}^{s,t}_p $ if
$$\begin{array}{l}
I\backslash\tilde I=\{s+1,s+2,\ldots s+\alpha\},\\
J\backslash\tilde J=\{t+1,t+2,\ldots t+\beta\}\end{array}$$
for some $\alpha$ and $\beta$.
\end{definition}\begin{corollary} \label{cor3}
Theorem \ref{theo4} is true with $T^{s+t}_p$ replaced by $\tilde{T}^{s,t}_p$. 
\end{corollary}\qed

\begin{lemma}\label{cor2}
Suppose $(I,J,K)\in \tilde{T}^{s,t}_p$ are as in definition \ref{up11}. Then 
$$\#\tilde{I}+\#\tilde J\ge p.$$

Moreover, suppose
$$\#\tilde{I}+\#\tilde J>p.$$
Then there exists $\gamma\notin K$ such that
$$(I\sqcup\{s+\alpha+1\},J\sqcup\{t+\beta+1\},K\sqcup\{\gamma\})\in T^{s+t}_{p+1}.$$
\end{lemma}
\begin{proof}
In case
$$\#\tilde{I}+\#\tilde J< p,$$
the triple $({\{\#\tilde{I}+1\},\{\#\tilde J+1\},\{\#\tilde{I}+\#\tilde J+1\}})$ belongs to $T^p_1$. Therefore,
$$i_{\#\tilde{I}+1}+j_{\#\tilde J+1}\le k_{\#\tilde{I}+\#\tilde J+1}+1$$
should be satisfied. However,
$$ k_{\#\tilde{I}+\#\tilde J+1}+1\le s+t+1<s+t+2=i_{\#\tilde{I}+1}+j_{\#\tilde J+1}.$$

Now suppose
$$\#\tilde{I}+\#\tilde J> p.$$
By the trace equality we obtain $\gamma=s+t+\alpha+\beta-p+1$. Then 
$$\#\tilde{I}+\#\tilde J=p-\alpha+p-\beta>p,$$
implies that $\gamma\le s+t$. Moreover, by Example \ref{dual},
$$k_p\le i_p+j_p-p=s+t+\alpha+\beta-p=\gamma-1,$$
and thus, $\gamma\notin K$.
Finally, matrices $A(I)\oplus(s+\alpha-p)$, $B(J)\oplus(t+\beta-p)$, $C(K)\oplus(s+t+\alpha+\beta-2p)$ satisfy Theorem \ref{theo3}, hence 
$$(I\sqcup\{s+\alpha+1\},J\sqcup\{t+\beta+1\},K\sqcup\{\gamma\})\in T^{s+t}_{p+1}.$$
\end{proof}

\begin{definition}
We say that $(I,J,K)\in  T^{s,t}_p $ if and only if $(I,J,K)\in  \tilde{T}^{s,t}_p $ and 
$$\#\tilde{I}+\#\tilde{J}=p.$$
\end{definition}
\begin{corollary}Theorem \ref{theo4} is tru\label{cor4}e with $T^{s+t}_p$ replaced by $T^{s,t}_p$. 
\end{corollary}
\begin{proof}
By Corollary \ref{cor3} we replace $T_p^{s+t}$ by $\tilde{T}^{s,t}_p$. We want to show that there is no essential  inequalities lying in $\tilde{T}^{s,t}_p\backslash T^{s,t}_p$. Consider a triple $(I,J,K)\in \tilde{T}^{s,t}_p$ such that
$\#\tilde{I}+\#\tilde{J}>p$.
Then by Theorem \ref{theo4} this triple  corresponds to the inequality
$$
\sum_{\tilde I} a_i+\sum_{\tilde J} b_j\ge\sum_{k\in K}c_k.
$$
This inequality is redundant since by Lemma \ref{cor2} there is a triple $(I\sqcup\{s+\alpha+1\},J\sqcup\{t+\beta+1\},K\sqcup\{\gamma\})\in T^{s+t}_{p+1}$ which corresponds to a stronger inequality
$$
\sum_{\tilde I} a_i+\sum_{\tilde J} b_j\ge\sum_{k\in K}c_k+c_\gamma.
$$

\end{proof}

\section{On groups of points}    	\label{compu}

Recall that we intend to study abelian threefolds with characteristic polynomials of the form
$$\leqno(\ref{problem}) \begin{array}{rlll}
\quad\quad(1)&f_X(t)=P^2(t)Q(t),&\deg P=\deg Q=2,&P(t)Q(t)\mbox{ separable};\\
(2)&f_X(t)=P(t)(t\pm\sqrt{q})^2,&\deg P=4,&P(t)(t\pm\sqrt{q})\mbox{ separable};\\
(3)&f_X(t)=Q^2(t)(t\pm\sqrt{q})^2,&\deg Q=2,&Q(t)(t\pm\sqrt{q})\mbox{ separable}.
\end{array}$$

The following general method is used. Suppose  $f_X(t)$  is decomposable over $\Z$  into a product of coprime factors $f_X(t)=P(t)Q(t)$, the variety $X$ can be included in an exact sequence of abelian varieties (as fppf sheaves)
$$0\longrightarrow Y\longrightarrow X\longrightarrow Z \longrightarrow 0$$
with $Y$ and $Z$ lying in isogeny classes corresponding to polynomials $P$ and $Q$. Thus there is also an exact sequence of Tate modules, and even of $\Z_l\left[T\right]$--modules 
$$0\longrightarrow \operatorname{T}_l(Y)\longrightarrow \operatorname{T}_l(X)\longrightarrow \operatorname{T}_l(Z) \longrightarrow 0,$$
with $T$ acting as $\operatorname{Fr}_X$, $\operatorname{Fr}_Y$, and $\operatorname{Fr}_Z$ on $X$, $Y$, and $Z$ respectively.

Thus if one fixes $Y$ and $Z$ and in particular,  $\Z_l\left[T\right]$--modules $\operatorname{T}_l(Y)$ and $\operatorname{T}_l(Z)$, then  by the results of Section  \ref{horn} Smith's invariants  of the module $\operatorname{T}_l(X)$ should satisfy the conditions of Theorem \ref{theo4}.

Conversely, let $Y$ and $Z$ be abelian varieties with coprime Weil polynomials $P$ and $Q$ and suppose there is an exact sequence of $\Z_l\left[T\right]$--modules 
$$0\longrightarrow \operatorname{T}_l(Y)\longrightarrow\mathcal{T}_l \longrightarrow \operatorname{T}_l(Z) \longrightarrow 0,$$
with $T$ acting as  $1-\operatorname{Fr}_Y$  on $\operatorname{T}_l(Y)$ and as $1-\operatorname{Fr}_Z$ on $ \operatorname{T}_l(Z)$. Since $P$ and $Q$ are coprime, $T$ acts semisimply on $\mathcal T_l$. Therefore one might construct an inclusion $\mathcal T_l\subset \operatorname{T}_l(Y\times Z)$ as a Frobenius invariant submodule.
Then there exists an abelian variety $X$ which is $l$--isogenuous to $(Y\times Z)$  such that $\mathcal T_l\cong \operatorname{T}_l(X)$ as $\Z_l\left[T\right]$--modules (see for example, \cite{Ry10}). Thus given $Y$ and $Z$ (in particular, $\Z_l\left[T\right]$--modules $\operatorname{T}_l(Y)$ and $\operatorname{T}_l(Z)$), and given a module $\mathcal{T}_l$ satisfying conditions of Theorem \ref{theo4} or in fact Corollary \ref{cor4}, there exists an abelian variety $X$ such that $\mathcal T_l\cong \operatorname{T}_l(X)$.

Therefore to list possible exponents of groups of points on abelian threefolds we should consider all possible pairs $Y$ and $Z$ and list varieties $X$ whose Smith's invariants satisfy Corollary \ref{cor4}.

Since  both $Y$ and $Z$ have dimensions at most $2$ we recall part of the classification of Tate modules of abelian varieties of small dimensions is given in \cite{Ry10}, \cite{Ry12}, and \cite{Ry15}.  Consider a variety $Y$ with  $\operatorname{dim} Y\le 2$ and   suppose $f_X(1-t)$ is the characteristic polynomial of the operator $1-Fr_Y$. If $f_X$ is separable, then by Theorems \ref{maintheo}  and \ref{maintheo2} groups of points of abelian varieties isgenuous to $X$ are exactly the groups $G$ whose Hodge polygon $\operatorname{Hp}(H, 2g)$ lies on or below $\operatorname{Np}(f_X(1-t))$.  Otherwise $f_X$ is a square and this case is treated by Theorem~\ref{groupsum}.

First, we intend to give an alternative proof of a stronger version of Theorem \ref{groupsum}.

\begin{theorem}
[\cite{Ry15}]
Let $f(t)$ be a Weil polynomial, $f(1-t)=P^rQ^s$ where $Q$ divides $P$, $\deg P= 2$ and $\deg Q=1$ and $P$ is separable. Then $G$ is the group of points on some abelian variety with Weil polynomial $f$ if and only if \begin{equation}\label{decomp}G\cong \oplus_{j=1}^r G_j\oplus H\end{equation} where $G_j$ are $l$-primary abelian groups such that $\operatorname{Np}_l(P(t))$ lies on or above $\operatorname{Hp}(G_j,2)$ for all $1\leq j\leq r$, and $H\cong(\Z_l/Q(0)\Z_l)^s$.
\end{theorem}

\begin{proof}
First,
%
the case $f(t)=(t\pm\sqrt{q})^s$ is trivial.

Otherwise, the statement is an easy consequence of the classification of finite torsionfree modules over orders with cyclic index presented below.  Let $\mathfrak{A}$ be a separable commutative  algebra of finite rank $n$ over~$\mathbb{Q}_l$, suppose $\mathcal O_{\mathfrak{A}}\subset \mathfrak A$ is a maximal order in $\mathfrak{A}$, and suppose $\mathcal{O}\subset\mathcal O_{\mathfrak{A}}\subset \mathfrak{A}$ is an \emph{order with cyclic index} that is $\mathcal O_{\mathfrak{A}}/\mathcal{O}$ is a cyclic group.

By a \emph{lattice} in $\mathfrak A$ we mean an additive subgroup  $M\subset \mathfrak A$ which is a free $\Z_l$-module of rank $n$. For a lattice $M\subset \mathfrak A$ the set \mbox{$\mathcal{O}=\{\xi\in \mathcal A \,|\, \xi M\subset M\}$} is an order and we say that the lattice $M$ \emph{belongs} to the order~$\mathcal O$.

\begin{theorem}
[\cite{BF}]\label{bf1}
Suppose $\mathcal{O}\subset \mathfrak{A}$ is an order with cyclic index in a separable commutative finite  $\Q_l$--algebra $\mathfrak{A}$ and let $M$ be  a finitely generated  torsionfree $\mathcal{O}$--module. Then there exists the chain of ideals of $\mathfrak{A}$ 
$$\mathfrak{A}\supseteq\mathcal{B}_1\supset\mathcal B_2\supset\ldots\supset\mathcal{B}_s,$$
such that in each algebra $\mathcal B_i$ there exists a lattice $B_i$ which belongs to a $\mathcal{B}_i$--order with cyclic index, and there exists an $\mathcal{O}$--decomposition
$$M\cong B_1\oplus B_2\oplus\ldots\oplus B_{s-1}\oplus B_s.$$
\end{theorem}

\begin{theorem}
[\cite{BF}]\label{bf2}
Suppose $\mathcal{O}\subset \mathfrak{A}$ is an order with cyclic index in a separable commutative finite  $\Q_l$--algebra $\mathfrak{A}$ and let $M\subset \mathfrak A^s$ be  a  module of rank $s $ over $\mathcal{O}$. Then there exists the chain of orders of $\mathfrak{A}$ 
$$\mathcal{O}\subseteq\mathcal{O}_1\subset\mathcal O_2\subset\ldots\subset\mathcal{O}_s\subseteq\mathcal O_{\mathfrak{A}}$$
and the decomposition
$$M\cong\mathcal{O}_1\oplus\mathcal O_2\oplus\ldots\oplus\mathcal{O}_{s-1}\oplus \mathcal M_s$$
where $\mathcal M_s\subset \mathfrak A$ is a lattice which belongs to $\mathcal O_s$. Equivalently, there is a decomposition 
$$M\cong M_1\oplus M_2\oplus\ldots\oplus M_s$$
where $M_i\subset \mathfrak A$  are lattices which belong to $\mathcal O_i$. 
\end{theorem}

Recall that the Frobenius morphism acts separably on the Tate  module $\operatorname{T}_l(X)$. Therefore $\operatorname{T}_l(X)$ is a module over an order~$\Z_l[1-\operatorname{Fr}_X]$  in an algebra $\mathbb Q_l[T]/P(T)$, which is separable over $\mathbb Q_l$. The algebra $\mathbb Q_l[T]/P(T)$ is isomorphic either to a quadratic extension $K$ of $\mathbb Q_l$ or to $\mathbb Q_l\times \mathbb Q_l$. In both cases all orders $\mathcal O\subset \mathbb Q_l[T]/P(T)$ are with cyclic index.

Therefore $\operatorname{T}_l(X)$ is a $\Z_l[1-\operatorname{Fr}_X]$--module as in Theorem \ref{bf1}. There is a  decomposition into a sum of two  $\Z_l[1-\operatorname{Fr}_X]$--modules
$$\operatorname{T}_l(X)\cong B_1\oplus B_2,$$
and to each $B_i$  Theorem \ref{bf2} applies. Thus
$$\operatorname{T}_l(X)\cong M_1\oplus M_2\oplus\ldots\oplus M_r\oplus\Z_l^{s}.$$

$$\operatorname{coker}\left( \operatorname{T}_l(X)\xrightarrow{1-\operatorname{Fr}_X}\operatorname{T}_l(X)\right)
\cong\oplus_{i=1}^{r}
\operatorname{coker}\left(  M_i\xrightarrow{1-\operatorname{Fr}_X} M_i\right)
\oplus
\operatorname{coker}\left(\Z_l^s\xrightarrow{1-\operatorname{Fr}_X}\Z_l^s\right).$$

The operator $(1-\operatorname{Fr}_X)$ acts on the modules $M_i$ with characteristic polynomial equal to $P(t)$. Both polynomials are separable,  thus  Theorems \ref{maintheo} and \ref{maintheo2} apply. Therefore a group $G_i$ is equal to 
$$\operatorname{coker}\left(M_i\xrightarrow{1-\operatorname{Fr}_X}M_i\right)$$
for some lattice $M_i$ if and only if their Hodge polygons $\operatorname{Hp}(G_i,2)$ lie on or below Newton polygon $\operatorname{Np}(P(t))$.

The operator $(1-\operatorname{Fr}_X)$ acts on the modules $\Z_l^s$ with characteristic polynomial equal to $Q(t)$, that is $(1-\operatorname{Fr}_X)$ is multiplication by $Q(0)$. Therefore
$$H\cong\operatorname{coker}\left(\Z_l^s\xrightarrow{1-\operatorname{Fr}_X}\Z_l^s\right)\cong(\Z_l/Q(0)\Z_l)^s$$
as desired.

\end{proof}

\subsection{}   \label{dif ficult}
Let $f_X(t)=P^2(t)Q(t)$ with $\deg P=\deg Q=2$, $PQ$ separable. Than there is an exact sequence 
$$0\longrightarrow Y\longrightarrow X\longrightarrow Z \longrightarrow 0$$
where $Y$ and $Z$ are associated with polynomials $P^2$ and $Q$ respectively. Suppose  groups $Y(\F_q)_l$ and $Z(\F_q)_l$ have exponents $\mathfrak{a}=(a_1,a_2,a_3,a_4)$ and $\mathfrak{b}=(b_1,b_2)$ respectively. Suppose $\mathfrak{c}=(c_1,\ldots ,c_6)$ are exponents of $X(\F_q)_l$. Sets $\mathfrak{a}$ and $\mathfrak{b}$ depend on $P$ and $Q$ and satisfy the conditions of Theorems \ref{maintheo} and \ref{groupsum}.  We shall list the inequalities which by Theorem \ref{theo4} relate $\mathfrak c$ to $\mathfrak a$ and $\mathfrak b$.

By  Corollary \ref{cor4} the essential inequalities are encoded in triples $(I,J,K)\in T^{4,2}_p$. There are only $4$ possibilities for $\tilde J=J\cap M_2$, namely $\emptyset$, $\{1\}$, $\{2\}$, $\{1,2\}$. We use complementary inequalities instead of inequalities of size at least $4$.

The inequalities with $p=1$ and $p=5$ are listed separately due to their importance. 
\setcounter{ineqone}{0}

{\it1: } Inequalities of size $1$ and $5$. Those are
\begin{equation}
\begin{array}{rrrcll}
\left[0\right]&(I=\{i\}, J=\{1\})&\quad a_i+b_1&\ge& c_i,&\qquad 1\le i \le 4\label{4}\\
\ione&(I=\{i\}, J=\{2\})&\quad a_i+b_2&\ge& c_{i+1},&\qquad 1\le i \le 4\\
 \ione&(I=\{5\}, J=\{1\})&b_1&\ge&c_5&\\
 \ione&(I=\{5\}, J=\{2\})&b_2&\ge&c_6&\\
 \ione&(I=\{i\}, J=\{3\})&a_i&\ge&c_{i+2}, &\qquad 1\le i \le 4\\
 \ione&(I=M_6\backslash\{6\}, J=M_6\backslash\{1\})&b_1&\le&c_1&\\
 \ione&(I=M_6\backslash\{6\}, J=M_6\backslash\{2\})&b_2&\le&c_2&\\
 \ione&(I=M_6\backslash\{i\},J=M_6\backslash\{6\})&a_i&\le&c_i, &\qquad 1\le i \le 4.
\end{array}
\end{equation}

{\it 2: } The set $\tilde J=\{1,2\}$ corresponds to Lidskii inequalities, see Example \ref{lid1}.  The complementary inequalities  have the form
$$\sum_{i\in I} a_i\le \sum_{i\in I} c_i,$$
and follow from \ref{4}[6].

By Example \ref{lid2}, the inequalities with $\tilde J=\emptyset$ have the form
$$\sum_{i\in I} a_i\ge \sum_{i\in I} c_{i+2},$$
and follow from \ref{4}[3].

{\it3: } Let  $\tilde J=\{2\}$. Then by Example \ref{lid2}, we have $k_\alpha=i_\alpha +1$ for every $\alpha$. The corresponding inequalities are
\begin{equation}\label{5}
\ione\qquad I\subset M_4\qquad \sum_{i\in I}a_i+b_2\ge\sum_{i\in I}c_{i+1}+c_6
\end{equation}

{\it4: } Let  $\tilde J=\{1\}$, that is $J=\{1,3,4,\ldots,p+1\}$.  Then for all $1\le \alpha\le p$ we have
$$\begin{array}{l}
i_\alpha+j_p\ge k_\alpha+p,\\
i_\alpha+j_1\le k_\alpha+1,
\end{array}$$
and therefore $i_\alpha\le k_\alpha\le i_\alpha+1$.
By 
$$\sum_I i+\sum_J j=\sum_K k,$$
there must be only one $\beta$ such that $i_\beta=k_\beta$ and $i_\alpha+1 =k_\alpha$ for $\alpha\ne\beta$. If these conditions are satisfied then we may set 
$$\begin{array}{l}
A(I)=\operatorname{diag}(i_1-1, \ldots i_p-p),\\
B(J)=\operatorname{diag}(1,\ldots,1,0, 1,\ldots,1),\\
C(K)=\operatorname{diag}(k_1-1, \ldots k_p-p),
\end{array}$$ 
with $B(J)_{\beta,\beta}=0$. These matrices satisfy Theorem \ref{theo3}, therefore $(I,J,K)\in T^n_p$. 

Since we only need triples $(I,J,K)\in T^{s,t}_p$, it suffices to consider sets $I$ containing $5$ and not containing $6$. Indeed, $\# I\backslash M_4=p-\# I\cap M_4=\# J\cap M_2=1$. The corresponding  inequalities of size $2$ and $4$ are
\begin{equation}
\begin{array}{rrrcll}
\ione&(I=\{i,5\}, K=\{i,6\})&\quad a_i+b_1&\ge& c_i+c_6,&\qquad 1\le i \le 4\label{6}\\
\ione&(I=\{i,5\}, K=\{i+1,5\})&a_i+b_1&\ge&c_{i+1}+c_5,&\qquad 1\le i \le 3\\\
\ione&(I=M_5\backslash\{i\}, K=M_6\backslash\{1,(i+2)\})&a_i+b_2&\le&c_1+c_{i+2},&\qquad 1\le i \le 4\\
\ione&(I=M_5\backslash\{i\},K=M_6\backslash\{2,(i+1)\})&a_i+b_2&\le&c_2+c_{i+1}, &\qquad 1\le i \le 3.
\end{array}
\end{equation}
Note that  \ref{5}[7] for $\#I=1$  and \ref{6}[8]  imply \ref{4}[0], as predicted by Corollary \ref{cor4}.

The inequalities of size $3$ are
\begin{equation}
\begin{array}{rrrcll}
 \ione&(I=\{1,2,5\}, K=\{1,3,6\})&\quad a_1+a_2+b_1&\ge& c_1+c_3+c_6\label{7}\\
 \ione&(I=\{1,2,5\}, K=\{2,3,5\})&a_1+a_2+b_1&\ge&c_2+c_3+c_5\\
 \ione&(I=\{1,3,5\}, K=\{1,4,6\})&a_1+a_3+b_1&\ge&c_1+c_4+c_6\\\addtocounter{ineqone}{-1}
\ione&(I=\{1,3,5\}, K=\{2,3,6\})&a_1+a_3+b_1&\ge&c_2+c_3+c_6\\
\ione&(I=\{1,3,5\}, K=\{2,4,5\})&a_1+a_3+b_1&\ge&c_2+c_4+c_5\\
\ione&(I=\{1,4,5\}, K=\{1,5,6\})&a_1+a_4+b_1&\ge&c_1+c_5+c_6\\
\ione&(I=\{1,4,5\}, K=\{2,4,6\})&a_1+a_4+b_1&\ge&c_2+c_4+c_6\\
\ione&(I=\{2,3,5\}, K=\{2,4,6\})&a_2+a_3+b_1&\ge&c_2+c_4+c_6\\
\ione&(I=\{2,3,5\}, K=\{3,4,5\})&a_2+a_3+b_1&\ge&c_3+c_4+c_5\\
\ione&(I=\{2,4,5\}, K=\{2,5,6\})&a_2+a_4+b_1&\ge&c_2+c_5+c_6\\
\ione&(I=\{2,4,5\}, K=\{3,4,5\})&a_2+a_4+b_1&\ge&c_3+c_4+c_5\\
\ione&(I=\{3,4,5\}, K=\{3,5,6\})&a_3+a_4+b_1&\ge&c_3+c_5+c_6\\
\end{array}
\end{equation}

Inequalities \ref{7}[14] are equivalent due to Theorem \ref{groupsum}.

\subsection{}    \label{easy}
Let $f_X(t)=P(t)(t\pm\sqrt{q})^2$, $P(t)$ separable and $P(\mp\sqrt{q})\ne 0$ or $P(t)=Q^2(t)$ with $Q(t)$ separable and $Q(\mp\sqrt{q})\ne 0$. Than there is an exact sequence 
$$0\longrightarrow Y\longrightarrow X\longrightarrow Z \longrightarrow 0,$$
where $Y$ and $Z$ are associated with polynomials $P(t)$ and $(t\pm\sqrt{q})^2$ respectively. Note that $1-\operatorname{Fr}_Z$ acts as multiplication by $b=v_l(1\pm q)$, so its Smith's invariants are $\mathfrak{b}=(b,b)$. Suppose $Y(\F_q)_l$  and $X(\F_q)_l$ have exponents $\mathfrak{a}=(a_1,a_2,a_3,a_4)$ and  $\mathfrak{c}=(c_1,\ldots ,c_6)$ respectively. Sets $\mathfrak a$ and $\mathfrak b$ satisfy Theorem \ref{maintheo}. Moreover, in case $P=Q^2$ Theorem \ref{groupsum} is applicable, and therefore $a_1+a_4=a_2+a_3$. 
Again we list inequalities relating  $\mathfrak{c}$ to $\mathfrak{a}$  and~$\mathfrak{b}$.

To list the inequalities we  only consider triples $(I,J,K)\in T^{4,2}_p$ with $\tilde J$ one of $\emptyset$, $\{1\}$, $\{2\}$, $\{1,2\}$. Since $\mathfrak b=(b,b)$, by Lemma \ref{cor1}, we don't need to consider sets $J$ containing $2$ but not $1$. Thus the case $\tilde J=\{2\}$ is now redundant. 

{\it 1: } Inequalities of size $1$ and $5$. Those are

\setcounter{ineqone}{0}

\begin{equation}
\begin{array}{rrrcll}
\ione&(I=\{i\}, J=\{1\})&\quad a_i+b&\ge& c_i,&\qquad 1\le i \le 4\label{1}\\
 \ione&(I=\{5\}, J=\{1\})&b&\ge&c_5&\\
 \ione&(I=\{i\}, J=\{3\})&a_i&\ge&c_{i+2}, &\qquad 1\le i \le 4\\
 \ione&(I=M_6\backslash\{6\}, J=M_6\backslash\{2\})&b&\le&c_2&\\
 \ione&(I=M_6\backslash\{i\},J=M_6\backslash\{6\})&a_i&\le&c_i, &\qquad 1\le i \le 4.
\end{array}
\end{equation}

{\it 2: } For $\tilde J=\{1,2\}$ the complementary inequalities  have the form
$$\sum_{i\in I} a_i\le \sum_{i\in I} c_i,$$
and follow from \ref{1}[4] (see Example \ref{lid1}).

By Example \ref{lid2}, the inequalities with $\tilde J=\emptyset$ have the form
$$\sum_{i\in I} a_i\ge \sum_{i\in I} c_{i+2},$$
and follow from \ref{1}[2].

{\it 3: } The case  $\tilde J=\{1\}$ is treated in \ref{dif ficult} part {\it 4}. 
In this situation $I$ contains $5$ and  does not contain $6$. Moreover,  $i_\beta=k_\beta$ for some $1\le \beta\le 4$ and $i_\alpha+1 =k_\alpha$ for $\alpha\ne\beta$.  The corresponding  inequalities of sizes $2$ and $4$ are
\begin{equation}
\begin{array}{rrrcll}
\ione&(I=\{i,5\}, K=\{i,6\})&\quad a_i+b&\ge& c_i+c_6,&\qquad 1\le i \le 4\label{2}\\
\ione&(I=\{i,5\}, K=\{i+1,5\})&a_i+b&\ge&c_{i+1}+c_5,&\qquad 1\le i \le 3\\\
\ione&(I=M_5\backslash\{i\}, K=M_6\backslash\{1,(i+2)\})&a_i+b&\le&c_1+c_{i+2},&\qquad 1\le i \le 4\\
\ione&(I=M_5\backslash\{i\},K=M_6\backslash\{2,(i+1)\})&a_i+b&\le&c_2+c_{i+1}, &\qquad 1\le i \le 3.
\end{array}
\end{equation}
Note that \ref{2}[5] imply \ref{1}[0].

The inequalities of size $3$ are
\begin{equation}
\begin{array}{rrrcll}
 \ione&(I=\{1,2,5\}, K=\{1,3,6\})&\quad a_1+a_2+b&\ge& c_1+c_3+c_6\label{3}\\
 \ione&(I=\{1,2,5\}, K=\{2,3,5\})&a_1+a_2+b&\ge&c_2+c_3+c_5\\
 \ione&(I=\{1,3,5\}, K=\{1,4,6\})&a_1+a_3+b&\ge&c_1+c_4+c_6\\
\ione&(I=\{1,3,5\}, K=\{2,3,6\})&a_1+a_3+b&\ge&c_2+c_3+c_6\\
\ione&(I=\{1,3,5\}, K=\{2,4,5\})&a_1+a_3+b&\ge&c_2+c_4+c_5\\
\ione&(I=\{1,4,5\}, K=\{1,5,6\})&a_1+a_4+b&\ge&c_1+c_5+c_6\\
\ione&(I=\{1,4,5\}, K=\{2,4,6\})&a_1+a_4+b&\ge&c_2+c_4+c_6\\
\ione&(I=\{2,3,5\}, K=\{2,4,6\})&a_2+a_3+b&\ge&c_2+c_4+c_6\\
\ione&(I=\{2,3,5\}, K=\{3,4,5\})&a_2+a_3+b&\ge&c_3+c_4+c_5\\
\ione&(I=\{2,4,5\}, K=\{2,5,6\})&a_2+a_4+b&\ge&c_2+c_5+c_6\\
\ione&(I=\{2,4,5\}, K=\{3,4,5\})&a_2+a_4+b&\ge&c_3+c_4+c_5\\
\ione&(I=\{3,4,5\}, K=\{3,5,6\})&a_3+a_4+b&\ge&c_3+c_5+c_6\\
\end{array}
\end{equation}

In case $P=Q^2$ inequalities \ref{3}[15] and \ref{3}[16] are identical.

\newpage

\section{Main result}\label{mainresult}
Computation conducted in the previous section combined with Theorems \ref{maintheo} and \ref{groupsum} imply the following three results.			
\begin{theorem}
Suppose $X/\F_q$ is an abelian threefold with characteristic polynomial of Frobenius map equal to $f_X(t)=P^2(t)Q(t)$ with $\deg P=\deg Q=2$, $PQ$ separable, suppose $(m_1\ge m_2)$ and $(n_1\ge n_2)$ are absolute values in $\Z_l$ of zeros of $P(1-t)$ and $Q(1-t)$ respectively. Then there exists an isogenous variety $\tilde X$ with group of points 
$$\tilde X(\F_q)_l=\Z/l^{c_1}\oplus\Z/l^{c_2}\oplus\ldots \oplus\Z/l^{c_6}$$ if and only if there exist sets of numbers $\mathfrak{a}=(a_1,a_2,a_3,a_4)$, $\mathfrak{b}=(b_1,b_2)$ such that
\begin{enumerate}
\item$c_1+\ldots+c_6=2m_1+2m_2+n_1+n_2$
\item $a_2\le a_1\le m_1$\\
$a_1+a_4=a_2+a_3=m_1+m_2$\\
$n_2\le b_2\le b_1\le n_1$\\
$b_1+b_2=n_1+n_2$\setcounter{ineqone}{1}
\item the following inequalities hold$$
\begin{array}{crcll}
\ione      &b_1&\ge&c_5&\\
\ione      &b_2&\ge&c_6&\\
\ione      &a_i&\ge&c_{i+2}, &\qquad 1\le i \le 4\\
\ione      &b_1&\le&c_1&\\
\ione      &b_2&\le&c_2&\\
\ione      &a_i&\le&c_i, &\qquad 1\le i \le 4\\
\ione      &\sum_{i\in I}a_i+b_2&\ge&\sum_{i\in I}c_{i+1}+c_6&\qquad I\subset M_4\\
\ione      & a_i+b_1&\ge& c_i+c_6,&\qquad 1\le i \le 4\\
\ione      &a_i+b_1&\ge&c_{i+1}+c_5,&\qquad 1\le i \le 3\\\
\ione      &a_i+b_2&\le&c_1+c_{i+2},&\qquad 1\le i \le 4\\
\ione      &a_i+b_2&\le&c_2+c_{i+1}, &\qquad 1\le i \le 3\\
\ione      &\quad a_1+a_2+b_1&\ge& c_1+c_3+c_6\\
\ione      &a_1+a_2+b_1&\ge&c_2+c_3+c_5\\
\ione      &a_1+a_3+b_1&\ge&c_2+c_3+c_6\\
\ione      &a_1+a_3+b_1&\ge&c_2+c_4+c_5\\
\ione      &a_1+a_4+b_1&\ge&c_1+c_5+c_6\\
\ione      &a_1+a_4+b_1&\ge&c_2+c_4+c_6\\
\ione      &a_2+a_3+b_1&\ge&c_2+c_4+c_6\\
\ione      &a_2+a_3+b_1&\ge&c_3+c_4+c_5\\
\ione      &a_2+a_4+b_1&\ge&c_2+c_5+c_6\\
\ione      &a_2+a_4+b_1&\ge&c_3+c_4+c_5\\
\ione      &a_3+a_4+b_1&\ge&c_3+c_5+c_6
\end{array}$$
\end{enumerate}
\end{theorem}

\newpage


\setcounter{ineqone}{1}

\begin{theorem}
Suppose $X/\F_q$ is an abelian threefold with characteristic polynomial of Frobenius map equal to $f_X(t)=P(t)(t\pm\sqrt{q})^2$ with $P(t)\cdot(t\pm\sqrt{q})$ separable, suppose $(m_1\ge m_2\ge m_3\ge m_4)$ are absolute values in $\Z_l$ of zeros of $P(1-t)$ and $v_l(1\pm\sqrt{q})=b$. Then there exists an isogenous variety $\tilde X$ with group of points 
$$\tilde X(\F_q)_l=\Z/l^{c_1}\oplus\Z/l^{c_2}\oplus\ldots \oplus\Z/l^{c_6}$$ if and only if there exists a set of numbers $\mathfrak{a}=(a_1,a_2,a_3,a_4)$ such that
\begin{enumerate}
\item $c_1+\ldots+c_6=m_1+\ldots +m_4+2b$
\item $a_1\le m_1$\\
$a_1+a_2\le m_1+m_2$\\
$a_1+a_2+a_3\le m_1+m_2+m_3$\\
$a_1+a_2+a_3+a_4=m_1+m_2+m_3+m_4$ \setcounter{ineqone}{1}
\item the following inequalities hold$$
\begin{array}{crcll}
\ione   &b&\ge&c_5&\\
\ione  &a_i&\ge&c_{i+2}, &\qquad 1\le i \le 4\\
\ione  &b&\le&c_2&\\
\ione  &a_i&\le&c_i, &\qquad 1\le i \le 4\\
\ione  &\quad a_i+b&\ge& c_i+c_6,&\qquad 1\le i \le 4\\
\ione  &a_i+b&\ge&c_{i+1}+c_5,&\qquad 1\le i \le 3\\
\ione  &a_i+b&\le&c_1+c_{i+2},&\qquad 1\le i \le 4\\
\ione  &a_i+b&\le&c_2+c_{i+1}, &\qquad 1\le i \le 3\\
\ione  &\quad a_1+a_2+b&\ge& c_1+c_3+c_6\\
\ione  &a_1+a_2+b&\ge&c_2+c_3+c_5\\
\ione  &a_1+a_3+b&\ge&c_1+c_4+c_6\\
\ione  &a_1+a_3+b&\ge&c_2+c_3+c_6\\
\ione  &a_1+a_3+b&\ge&c_2+c_4+c_5\\
\ione  &a_1+a_4+b&\ge&c_1+c_5+c_6\\
\ione  &a_1+a_4+b&\ge&c_2+c_4+c_6\\
\ione  &a_2+a_3+b&\ge&c_2+c_4+c_6\\
\ione  &a_2+a_3+b&\ge&c_3+c_4+c_5\\
\ione  &a_2+a_4+b&\ge&c_2+c_5+c_6\\
\ione  &a_2+a_4+b&\ge&c_3+c_4+c_5\\
\ione  &a_3+a_4+b&\ge&c_3+c_5+c_6
\end{array}$$
\end{enumerate}
\end{theorem}

\newpage


\setcounter{ineqone}{1}

\begin{theorem}
Suppose $X/\F_q$ is an abelian threefold with characteristic polynomial of Frobenius map equal to $f_X(t)=P^2(t)(t\pm\sqrt{q})^2$ , $P(t)\cdot(t\pm\sqrt{q})$ separable, suppose $(m_1\ge m_2)$ are absolute values in $\Z_l$ of zeros of $P(1-t)$ and $v_l(1\pm\sqrt{q})=b$. Then there exists an isogenous variety $\tilde X$ with group of points 
$$\tilde X(\F_q)_l=\Z/l^{c_1}\oplus\Z/l^{c_2}\oplus\ldots \oplus\Z/l^{c_6}$$ if and only if there exists a set of numbers $\mathfrak{a}=(a_1,a_2,a_3,a_4)$ such that
\begin{enumerate}
\item$c_1+\ldots+c_6=2m_1+2m_2+2b$
\item $a_2\le a_1\le m_1$\\
$a_1+a_4=a_2+a_3=m_1+m_2$\\
\setcounter{ineqone}{1}
\item the following inequalities hold$$
\begin{array}{rrcll}
 \ione   &b&\ge&c_5&\\
\ione    &a_i&\ge&c_{i+2}, &\qquad 1\le i \le 4\\
\ione    &b&\le&c_2&\\
\ione    &a_i&\le&c_i, &\qquad 1\le i \le 4\\
\ione    &\quad a_i+b&\ge& c_i+c_6,&\qquad 1\le i \le 4\\
\ione    &a_i+b&\ge&c_{i+1}+c_5,&\qquad 1\le i \le 3\\
\ione    &a_i+b&\le&c_1+c_{i+2},&\qquad 1\le i \le 4\\
\ione    &a_i+b&\le&c_2+c_{i+1}, &\qquad 1\le i \le 3\\
\ione    &\quad a_1+a_2+b&\ge& c_1+c_3+c_6\\
\ione    &a_1+a_2+b&\ge&c_2+c_3+c_5\\
\ione    &a_1+a_3+b&\ge&c_2+c_3+c_6\\
\ione    &a_1+a_3+b&\ge&c_2+c_4+c_5\\
\ione    &a_1+a_4+b&\ge&c_1+c_5+c_6\\
\ione    &a_1+a_4+b&\ge&c_2+c_4+c_6\\ 
\ione ,[16]   &a_2+a_3+b&\ge&c_2+c_4+c_6\\ \addtocounter{ineqone}{1}
\ione    &a_2+a_3+b&\ge&c_3+c_4+c_5\\
\ione    &a_2+a_4+b&\ge&c_2+c_5+c_6\\
\ione    &a_2+a_4+b&\ge&c_3+c_4+c_5\\
\ione    &a_3+a_4+b&\ge&c_3+c_5+c_6
\end{array}$$
\end{enumerate}
\end{theorem}




\end{document}